\documentclass[12pt]{amsart}
\usepackage{amsmath,amssymb}
\usepackage{amsfonts}
\usepackage{amsthm}
\usepackage{latexsym}
\usepackage{graphicx}

\def\lf{\left}
\def\ri{\right}

\def\dbar{\bar\partial}

\def\R{\mathbb{R}}

\def\vv<#1>{\langle#1\rangle}
\def\G{\Gamma}

\def\bj{{\bar j}}
\def\bk{{\bar k}}

\def\bl{{\bar l}}

\def\bmu{{\bar\mu}}
\def\lam{{\lambda}}

\def\a{{\alpha}}

\def\be{{\beta}}

\def\om{{\omega}}

\def\XXint#1#2{\setbox0=\hbox{$#1{#2}{\int}$}{#2}\kern-.5\wd0 }

\def\XXint#1#2#3{{\setbox0=\hbox{$#1{#2#3}{\int}$}
     \vcenter{\hbox{$#2#3$}}\kern-.5\wd0}}

\def\vv<#1>{\lf\langle#1\ri\rangle}
\def\ol{\overline}
\newtheorem{thm}{Theorem}[section]

\newtheorem{lem}{Lemma}[section]

\newtheorem{cor}{Corollary}[section]
\theoremstyle{definition}
\newtheorem{defn}{Definition}[section]
\theoremstyle{remark}

\numberwithin{equation}{section}
\begin{document}
\title{Curvature of almost Hermitian manifolds and applications}
\author{Chengjie Yu$^1$ }
\address{Department of Mathematics, Shantou University, Shantou, Guangdong, 515063, China}
\email{cjyu@stu.edu.cn}
\thanks{$^1$Research partially supported by the National Natural Science Foundation of
China (11001161),(10901072) and (11101106).}

\renewcommand{\subjclassname}{%
  \textup{2000} Mathematics Subject Classification}
\date{}
\begin{abstract}
In this paper, by introducing a notion of local quasi holomorphic frame, we obtain a curvature formula for almost Hermitian manifolds which is similar to that of Hermitian manifolds. Moreover, as applications of the curvature formula, we extend a result of H.S. Wu and a result of F. Zheng to almost Hermitian manifolds.
\end{abstract}
\maketitle\markboth{Chengjie Yu}{curvature of almost hermitian manifold }
\section{Introduction}
A triple $(M,J,g)$ is called an almost Hermitian manifold where $M$ is a smooth manifold of even dimension, $J$ is an almost complex structure on $M$ and $g$ is a Riemannian metric on $M$ that is $J$-invariant. There are several kinds of interesting connections on almost Hermitian manifolds (see \cite{g}). Among them, the Levi-Civita connection which is torsion free and compatible with the metric and the canonical connection which is compatible with the metric and complex structure with vanishing $(1,1)$-part of the torsion attracted the most attentions.

The geometry of almost Hermitian manifolds with respect to the Levi-Civita connection was studied by Gray (See \cite{Gray1,Gray2,Gray3}) and the other geometers  in the 70's of the last century. An important conjecture  in this line was raised by Goldberg (\cite{Goldberg}): An Einstein almost most K\"ahler manifold must be K\"ahler. Here, almost K\"ahler manifolds means almost Hermitian manifolds with the fundamental $\omega_g(X,Y)=g(JX,Y)$ closed. The conjecture was proved by Sekigawa \cite{Sekigawa} with the further assumption of nonnegative scalar curvature. The full conjecture is still open. One can consult the survey \cite{AD} for recent progresses of the conjecture.

The canonical connection in crucial in the study of the structure of nearly K\"ahler manifolds by Nagy \cite{Na1,Na2}.  In \cite{twy}, Tossati, Weinkove and Yau used the canonical connection other than the Levi-Civita connection to solve the Calabi-Yau equation on almost K\"ahler manifolds related to an  interesting and important program on the study of simplectic topology proposed by Donaldson \cite{Donaldson}. Later, in \cite{vt}, Tossati obtained Laplacian comparison, a Schwartz lemma for almost Hermitian manifolds which is a generalization of Yau's Schwartz lemma for Hermitian manifolds(See \cite{Yau}). Moreover, with the help of the generalized Laplacian comparison and Schwartz lemma, Tossati extended a result by Seshadri-Zheng \cite {SZ} on the nonexistence of complete Hermitian metrics with holomorphic bisectional curvature bounded between two negative constants and bounded torsion on a product of complex manifolds to a product of almost complex manifolds with almost Hermitian metrics. In \cite{FTY}, Fan, Tam and the author further weaker the curvature assumption of the result of Tossati and obtain the same conclusion which is also a generalization of a result of Tam-Yu \cite{TY}.

The canonical connection was first introduced by Ehresmann and Libermann in \cite{e}. It is a natural generalization of the Chern connection on Hermitian manifolds (See \cite{Chern}) and is more related to the almost complex structure.
When the complex structure is integrable, it is just the Chern connection. In this paper, by introducing a notion of local quasi holomorphic frame, we obtain a curvature formula of almost Hermitian manifolds similar to that of Hermitian manifolds. More precisely, we have the following result.
\begin{thm}
Let $(M,g,J)$ be an almost Hermitian manifold. Let $(e_1,\cdots,e_n)$ be
a local quasi holomorphic  frame at $p$. Then
\begin{equation}
R_{i\bj k\bl}(p)=-\ol {e_l} e_k(g_{i\bj})(p)+g^{\bmu\lam}e_k(g_{i\bmu})\ol {e_l}(g_{\lam\bj})(p).
\end{equation}
\end{thm}
As an application of the curvature formula on almost Hermitian manifolds, we extend a result of Wu \cite{Wu} to almost Hermitian manifolds.
\begin{thm}
Let $(M,J)$ be an almost complex manifold. Let $g,h$ be two almost Hermitian metrics
on $M$. Then
 $$R^{g+h}(X,\ol X,Y,\ol Y)\leq R^g(X,\ol X,Y,\ol Y)+R^h(X,\ol X,Y,\ol Y)$$
for any two $(1,0)$ vectors $X$ and $Y$ where $R^{g+h},R^g$ and $R^h$ denote the curvature tensors of the metrics $g+h, g$ and $h$ respectively.
\end{thm}

Another application of the curvature formula on almost Hermitian manifolds in paper is to give a classification of almost Hermitian metrics with nonpositive holomorphic bisectional curvature on a product of two compact almost complex manifolds which is a generalization of a result of Zheng \cite{Zheng} and a previous result of the author \cite{Yu}. More precisely, we obtain the following results.

\begin{thm}\label{thm-class-prod}
Let $M$ and $N$ be compact almost complex manifolds.
Let $\phi_1,\phi_2,\cdots, \phi_r$ be a basis of $H^{1,0}(M)$ and
$\psi_1,\psi_2,\cdots,\psi_s$ be a basis of $H^{1,0}(N)$. Then, for
any almost Hermitian metric $h$ on $M\times N$ with nonpositive holomorphic
bisectional curvature,
\begin{equation*}
\omega_h=\pi_1^{*}\omega_{h_1}+\pi_2^{*}\omega_{h_2}+\rho+\bar\rho
\end{equation*}
where $h_1$ and $h_2$ are almost Hermitian metrics on $M$ and  $N$ with
nonpositive holomorphic bisectional curvature respectively, $\pi_1$
and $\pi_2$ are natural projections from $M\times N$ to $M$ and from $M\times N$ to
$N$ respectively, and
\begin{equation*}
\rho=\sqrt{-1}\sum_{k=1}^{r}\sum_{l=1}^{s}a_{kl}\phi_k\wedge\ol{\psi_l}
\end{equation*}
with $a_{kl}$'s are complex numbers.
\end{thm}
\begin{cor}
Let $M$ and $N$ be two almost Hermitian manifolds with $\mathcal H(M)\neq\emptyset$ and $\mathcal H(N)\neq\emptyset$. Then
$$\mbox{codim}_\R(\mathcal{H}(M)\times
\mathcal{H}(N),\mathcal{H}(M\times N))=2\dim H^{1,0}(M)\cdot
\dim H^{1,0}(N).$$
\end{cor}
Here $H^{1,0}(M)$ means the space of holomorphic $(1,0)$-from on $M$ and $\mathcal{H}(M)$ is the collection of almost Hermitian metric on $M$ with nonpositive bisectional curvature. The corollary above implies that an almost Hermitian metric on a product of two compact almost complex manifolds with nonpositive holomorphic bisectional curvature must be a product metric if one of the compact almost complex manifold admits no nontrivial holomorphic $(1,0)$-form.

 Note that the local quasi holomorphic frame introduced in this paper is different with the the generalized normal holomorphic frame introduced by Vezzoni \cite{Vezzoni} since the background connection considered in \cite{Vezzoni} is the Levi-Civita connection and Vezzoni's generalized normal frame exists only on quasi K\"ahler manifolds. The introduced frame is also different with local holomorphic coordinate introduced in \cite{FTY}.
\section{Frames on almost complex manifolds}
Recall the following definition of holomorphic vector fields on almost complex manifolds which is a generalization of holomorphic vector fields on complex manifolds.
\begin{defn}[\cite{g}]
Let $(M,J)$ be an almost complex manifold, a $(1,0)$-vector field is said
to be pseudo holomorphic at the point $p\in M$ if $(L_{\ol X}Y)^{(1,0)}(p)=[\ol X,Y]^{(1,0)}(p)=0$  for any $(1,0)$ vector field $X$ on $M$, where $Z^{(1,0)}$ means the $(1,0)$ part of $Z$ and $L$ means Lie deriviative. If $Y$ is pseudo holomorphic all over $M$, we call $Y$ a holomorphic vector field.
\end{defn}
Note that the almost complex structure may not be integrable, so we can not expect an existence of local holomorphic vector fields in general. However, for any point $p\in M$, we can find a local $(1,0)$-frame that is pseudo holomorphic at $p$.

\begin{lem}\label{lem-exist-pseudo-holo}
Let $(M,J)$ be an almost complex manifold. Then, for any $p\in M$, there
is a $(1,0)$-frame $(e_1,e_2,\cdots,e_n)$ near $p$ such that $e_i$ is pseudo holomorphic at $p$ for all $i$. We call the $(1,0)$-frame $(e_1,e_2,\cdots,e_n)$ a local pseudo holomorphic frame at $p$.
\end{lem}
\begin{proof}
Let $(v_1,v_2,\cdots,v_n)$ be a local $(1,0)$-frame of $M$ near $p$ and $(e_1,e_2,\cdots,e_n)$ be another local $(1,0)$-frame of $M$ near $p$. Suppose that \begin{equation}
e_i=f_{ij}v_j,
\end{equation}
and
\begin{equation}
[\ol {v_j},v_i]^{(1,0)}=c_{i\bj}^k v_k,
\end{equation}
where $f_{ij}$'s are local smooth functions to be determined. Then
\begin{equation}
\begin{split}
[\ol{v_j},e_i]^{(1,0)}(p)=\ol {v_{j}}(f_{i\mu})(p)v_\mu(p)+f_{i\lam}(p)c_{\lam\bj}^\mu(p) v_\mu(p).
\end{split}
\end{equation}
If we choose $f_{ij}$ such that  that $f_{ij}(p)=\delta_{ij}$ and $\ol {v_{k}}(f_{ij})(p)=-c_{i\bk}^j(p)$ for all $i,j$ and $k$, then
\begin{equation}
[\ol {v_j},e_i]^{(1,0)}(p)=0
\end{equation}
for all $i$ and $j$. So,  $(e_1,e_2,\cdots,e_n)$ is a local $(1,0)$-frame on $M$ near $p$ such that $e_i$ is pseudo holomorphic at $p$ for all $i$.
\end{proof}
Due to the nonexistence of local holomorphic frame on an almost complex manifold in general, we have to introduce a notion of holomorphicity at a point that is better than pseudo holomorphic for further application.
\begin{defn}
Let $(M,J)$ be an almost complex manifold. A $(1,0)$-vector field $X$ is called quasi holomorphic at the point $p\in M$ if $[Z,[\ol Y,X]]^{(1,0)}(p)=0$ for any $(1,0)$-vector fields $Y$ and $Z$ that are pseudo holomorphic at $p$, and moreover, $X$ itself is also pseudo holomorphic at $p$.
\end{defn}
From the definition, it is easy to check that  a holomorphic vector field is quasi holomorphic all over $M$.
\begin{lem}\label{lem-criterion-quasi-holo}
Let $(M,J)$ be an almost complex manifold and $(e_1,e_2,\cdots,e_n)$ be a local pseudo holomorphic frame at $p\in M$. Let $X$ be a $(1,0)$ vector field on $M$ that is pseudo holomorphic at $p$. Then, $X$ is quasi holomorphic at $p$ if and only if $[e_i,[\ol{e_j},X]]^{(1,0)}(p)=0$ for all $i$ and $j$.
\end{lem}
\begin{proof}
Let $Y=Y_i e_i$ and $Z=Z_i e_i$ be two $(1,0)$-vector fields that are pseudo holomorphic at $p$. Then, it is clear that $\ol{v}(Y_i)(p)=\ol{v}(Z_i)(p)=0$ for any $(1,0)$ vector $v$ at $p$. Then
\begin{equation}
\begin{split}
&[Z,[\ol{Y},X]]^{(1,0)}(p)\\
=&[Z_ie_i,[\ol{Y},X]]^{(1,0)}(p)\\
=&Z_i[e_i,[\ol{Y},X]]^{(1,0)}(p)-([\ol Y,X](Z_i))(p)e_i(p)\\
=&Z_i[e_i,[\ol{Y_je_j},X]]^{(1,0)}(p)\\
=&Z_i[e_i,\ol{Y_j}[\ol{e_j},X]]^{(1,0)}(p)-Z_i[e_i,X(\ol{Y_j})\ol{e_j}]^{(1,0)}(p)\\
=&Z_i\ol{Y_j}[e_i,[\ol{e_j},X]]^{(1,0)}(p)+Z_ie_i(\ol{Y_j})[\ol{e_j},X]^{(1,0)}(p)\\
=&Z_i\ol{Y_j}[e_i,[\ol{e_j},X]]^{(1,0)}(p)\\
=&0.
\end{split}
\end{equation}
\end{proof}
\begin{lem}\label{lem-exist-quasi-holo}
Let $(M,J)$ be an almost complex manifold. Then, for any $p\in M$, there
is a $(1,0)$-frame $(e_1,e_2,\cdots,e_n)$ near $p$ such that $e_i$ is quasi holomorphic at $p$ for all $i$. We call the $(1,0)$-frame $(e_1,e_2,\cdots,e_n)$ a local quasi holomorphic frame at $p$.
\end{lem}
\begin{proof}
Let $(v_1,v_2,\cdots,v_n)$ be a local pseudo holomorphic frame of $M$ at $p$ and $(e_1,e_2,\cdots,e_n)$ be a local $(1,0)$-frame of $M$ near $p$. Suppose that \begin{equation}
e_i=f_{ij}v_j,
\end{equation}
where $f_{ij}$'s are local smooth functions to be determined. We first assume that  $\ol{v_k}(f_{ij})(p)=0$ for all $i,j$ and $k$. Then, by the proof of Lemma \ref{lem-exist-pseudo-holo}, $(e_1,e_2,\cdots,e_n)$ is a local pseudo holomorphic frame at $p$.

Moreover, suppose that
\begin{equation}
[\ol {v_j},v_i]^{(1,0)}=c_{i\bj}^k v_k
\end{equation}
It is clear that $c_{i\bj}^k(p)=0$ for all $i,j$ and $k$ since $(v_1,v_2,\cdots,v_n)$ is a local pseudo holomorphic frame at $p$. Then
\begin{equation}
\begin{split}
&[v_k,[\ol{v_j},e_i]]^{(1,0)}(p)\\
=&[v_k,[\ol{v_j},f_{i\lam}v_\lam]]^{(1,0)}(p)\\
=&[v_k,f_{i\lam}[\ol{v_j},v_\lam]]^{(1,0)}(p)+[v_k,\ol{v_j}(f_{i\lam})v_\lam]^{(1,0)}(p)\\
=&f_{i\lam}[v_k,[\ol{v_j},v_\lam]]^{(1,0)}(p)+v_k(f_{i\lam})[\ol{v_j},v_\lam]^{(1,0)}(p)+\ol{v_j}(f_{i\lam})(p)[v_k,v_\lam]^{(1,0)}(p)\\
&+v_k\ol{v_j}(f_{i\lam})v_\lam(p)\\
=&f_{i\lam}[v_k,[\ol{v_j},v_\lam]]^{(1,0)}(p)+v_k\ol{v_j}(f_{i\lam})v_\lam(p)\\
=&f_{i\lam}[v_k,[\ol{v_j},v_\lam]^{(1,0)}]^{(1,0)}(p)+v_k\ol{v_j}(f_{i\lam})v_\lam(p)\\
=&f_{i\lam}[v_k,c_{\lam\bj}^\mu v_\mu]^{(1,0)}(p)+v_k\ol{v_j}(f_{i\lam})v_\lam(p)\\
=&f_{i\lam}c_{\lam\bj}^\mu[v_k,v_\lam]^{(1,0)}(p)+f_{i\lam}v_k(c_{\lam\bj}^\mu)v_\mu(p)+v_k\ol{v_j}(f_{i\lam})v_\lam(p)\\
=&(f_{i\lam}v_k(c_{\lam\bj}^\mu)+v_k\ol{v_j}(f_{i\mu}))v_\mu(p)\\
\end{split}
\end{equation}
If we further choose $f_{ij}$ such that $f_{ij}(p)=\delta_{ij}$ and
\begin{equation}
v_l\ol{v_k}(f_{ij})(p)=-v_{l}(c_{i\bk}^j)(p),
\end{equation}
for all $i,j,k$ and $l$, then
\begin{equation}
[v_k,[\ol{v_j},e_i]]^{(1,0)}(p)=0
\end{equation}
for all $i,j$ and $k$. By Lemma \ref{lem-criterion-quasi-holo}, we know that $e_i$ is quasi holomorphic at $p$ for all $i$. This completes the proof.
\end{proof}

\section{Frames on almost Hermitian manifolds}
In this section, we recall some basic definitions for almost Hermitian manifolds and introduce a notion for almost Hermitian manifolds that is analogous to normal frame on Hermitian manifolds

\begin{defn}[\cite{k,k2,g}] Let $(M,J)$ be an almost complex manifold. A Riemannian metric $g$ on $M$ such that $g(JX,JY)=g(X,Y)$ for any two tangent vectors $X$ and $Y$ is called an almost Hermitian metric.  The triple $(M,J,g)$ is called an almost Hermitian manifold. The two form $\omega_g=g(JX,Y)$ is called the fundamental form of the almost Hermitian manifold. A connection $\nabla$ on an almost Hermitian manifold $(M,J,g)$ such that $\nabla g=0$ and $\nabla J=0$ is called an almost Hermitian connection.
\end{defn}

For a connection $\nabla$ on a manifold $M$, recall that the torsion $\tau$ of the connection is a vector-valued two form defined as
\begin{equation}
\tau(X,Y)=\nabla_XY-\nabla_YX-[X,Y].
\end{equation}
On an almost Hermitian manifold, there are many almost Hermitian connections. However, there is a unique one such that $\tau(X,\overline Y)=0$ for any two $(1,0)$-vectors $X$ and $Y$. Such a notion is first introduced by Ehresman and Libermann \cite{e}.
\begin{defn}[\cite{k,k2}]The unique almost Hermitian connection $\nabla$ on an almost Hermitian manifold $(M,J,g)$ with vanishing $(1,1)$-part of the torsion is called the canonical connection of the almost Hermitian manifold.
\end{defn}

In this paper, for an almost Hermitian metric, the connection is always chosen to be the canonical connection $\nabla$. Recall that the curvature tensor $R$ of the connection $\nabla$ is defined as
\begin{equation}
R(X,Y,Z,W)=\vv<\nabla_Z\nabla_WX-\nabla_W\nabla_ZX-\nabla_{[Z,W]}X,Y>
\end{equation}
for any tangent vectors $X,Y,Z$ and $W$. Note that unlike the curvature tensor on Hermitian manifolds with Chern connection, the curvature tensor $R$ of the canonical connection may has non-vanishing (2,0)-part which mean that $R(X,\ol{Y}, Z,W)$ may not vanish for any $(1,0)$-vectors $X,Y,Z,W$. For the $(1,1)$-part of the curvature tensor $R$, we means $R(X,\ol Y,Z,\ol W)$ for $(1,0)$-vectors $X,Y,Z$ and $W$. Moreover, $R$ is said to be of nonpositive (negative) holomorphic bisectional curvature if $R(X,\ol X,Y,\ol Y)\leq 0(<0)$ for any two nonzero $(1,0)$-vectors $X$ and $Y$.

This notion of holomorphic vector fields introduced in the last section is somehow compatible with the canonical connection on almost Hermitian manifolds analogous to that on Hermitian manifolds.
\begin{lem}\label{lem-hol-con}
 $\nabla_{\ol X}Y=[\ol X,Y]^{(1,0)}$ for any two $(1,0)$-vector fields $X$ and $Y$ on an almost Hermitian manifold.
\end{lem}
\begin{proof}
By the definition of canonical connection and torsion, we have
\begin{equation}
\nabla_{\ol X}Y=\nabla_Y\ol X+[\ol X,Y]+\tau(\ol X,Y)=\nabla_Y{\ol X}+[\ol X,Y].
\end{equation}
Since $\nabla J=0$, we know that $\nabla_{\ol X}Y$ is a $(1,0)$-vector and $\nabla_Y\ol X$ is a $(0,1)$-vector. Therefore, the conclusion follows.

\end{proof}

A local pseudo holomorphic frame for an almost Hermitian manifold play a similar role as a local holomorphic frame for an Hermitian manifold. More precisely, we have the following formula for the Christoffel symbol of the canonical connection under a local pseudo holomorphic frame.

\begin{lem}\label{lem-christoffel}
Let $(M,J,g)$ be an almost Hermitian manifold and \\
$(e_1,e_2,\cdots,e_n)$ be a local pseudo holomorohic
frame at $p$. Then $\G_{i\bj}^k(p)=0$ and $\G_{ij}^k(p)=g^{\bmu k}e_j(g_{i\bmu})(p)$ for any $i,j$ and $k$.
\end{lem}
\begin{proof}
By Lemma \ref{lem-hol-con}, we know that $\nabla_{\ol e_j}e_i(p)=0$. So $\G_{i\bj}^k(p)=0$ for all $i,j$ and $k$. Moreover,
\begin{equation}
\begin{split}
&e_{j}(g_{i\bmu})(p)\\
=&\vv<\nabla_{e_j}e_i,\ol{e_\mu}>(p)+\vv<e_i,\nabla_{e_j}\ol{e_\mu}(p)>\\
=&\G_{ij}^k(p)g_{k\bmu}(p).
\end{split}
\end{equation}
Hence $\G_{ij}^k(p)=g^{\bmu k}e_j(g_{i\bmu})(p)$ for any $i,j$ and $k$.
\end{proof}

Similarly as on Hermitian manifolds, we can have a similar notion of normal holomorphic frame.
\begin{lem}\label{lem-pseudo-normal}
Let $(M,J,g)$ be an almost Hermitian manifold. Then, there is a local pseudo holomorphic frame $(e_1,e_2,\cdots,e_n)$ at $p$,
such that $\nabla e_i(p)=0$, or equivalently, $d g_{i\bj}(p)=0$ for all $i$ and $j$. We call the frame a local pseudo holomorphic normal frame at $p$.
\end{lem}

\begin{proof}
Let $(v_1,v_2,\cdots,v_n)$ be a local pseudo holomorphic frame at $p$ and $(e_1,\cdots,e_n)$
be a local $(1,0)$-frame at $p$. Suppose that
\begin{equation}
e_i=f_{ij}v_j,
\end{equation}
where $f_{ij}$'s are local smooth functions to be determined. We first assume that  $\ol{v_{k}}(f_{ij})(p)=0$ for all $i,j$ and $k$. Then, $(e_1,e_2,\cdots,e_n)$ is a local pseudo holomorphic coordinate at $p$.

Moreover, note that
\begin{equation}
\nabla_{v_k}e_i(p)=v_k(f_{i\mu})(p)v_\mu(p)+f_{ij}(p)\G_{jk}^\mu(p) v_\mu(p)
\end{equation}
where $\G_{jk}^\mu$ is the Christoffel symbol with respect to the frame $(v_1,v_2,\cdots,v_n)$. So, if we choose $f_{ij}$ such that $f_{ij}(p)=\delta_{ij}$, $\ol{v_k}(f_{ij})(p)=0$ and $v_k(f_{ij})(p)=-\G_{ik}^j(p)$. Then $\nabla e_i(p)=0$ for all $i$.
\end{proof}

By a similar argument as in the proof of Lemma \ref{lem-pseudo-normal}, we have the existence of a so called local quasi holomorphic normal frame.
\begin{lem}\label{lem-quasi-normal}
Let $(M,J,g)$ be an almost Hermitian manifold. Then, there is a local quasi holomorphic frame $(e_1,e_2,\cdots,e_n)$ at $p$,
such that $\nabla e_i(p)=0$, or equivalently, $d g_{i\bj}(p)=0$ for all $i$ and $j$. We call the frame a local quasi holomorphic normal frame at $p$.
\end{lem}
\section{Curvature of almost Hermitian manifolds}
In this section, we derive a formula for the curvature tensor with respect to a quasi holomorphic frame.

\begin{lem}\label{lem-part-curv}
Let $(M,g,J)$ be an almost Hermitian manifold and $(e_1,\cdots,e_n)$ be
a local quasi holomorphic frame at $p\in M$. Then
\begin{equation}
\nabla_{e_k}\nabla_{\ol{e_j}}e_i(p)=0
\end{equation}
for all $i,j$ and $k$.
\end{lem}
\begin{proof} First, note that
\begin{equation}\label{eqn-lie-hol}
[\ol{e_j},e_i](p)=\nabla_{\ol{e_j}}e_i(p)-\nabla_{e_i}\ol{e_j}(p)=0
\end{equation}
by Lemma \ref{lem-hol-con}. Moreover,
\begin{equation}
\begin{split}
&\vv<\nabla_{e_k}\nabla_{\ol{e_j}}e_i,\ol{e_l}>(p)\\
=&\vv<\nabla_{e_k}(\nabla_{e_i}{\ol{e_j}}+[\ol{e_j},e_i]),\ol{e_l}>(p)\\
=&\vv<\nabla_{e_k}[\ol{e_j},e_i],\ol{e_l}>(p)\\
=&\vv<\nabla_{[\ol{e_j},e_i]}{e_k}+[e_k,[\ol{e_j},e_i]]+\tau(e_k,[\ol{e_j},e_i]),\ol{e_l}>(p)\\
=&\vv<[e_k,[\ol{e_j},e_i]]^{(1,0)},\ol{e_l}>(p)\\
=&0.
\end{split}
\end{equation}
This completes the proof since $\nabla_{e_k}\nabla_{\ol{e_j}}e_i(p)$ is a $(1,0)$-vector.
\end{proof}
We are now ready to compute the $(1,1)$-part of the  curvature tensor for an almost Hermitian manifold.
\begin{thm}\label{thm-curv}
Let $(M,g,J)$ be an almost Hermitian manifold. Let $(e_1,\cdots,e_n)$ be
a local quasi holomorphic  frame at $p$. Then
\begin{equation}
R_{i\bj k\bl}(p)=-\ol{e_l} e_k(g_{i\bj})+g^{\bmu\lam}e_k(g_{i\bmu})\ol{e_l}(g_{\lam\bj}).
\end{equation}
\end{thm}
\begin{proof}
\begin{equation}
\begin{split}
&R_{i\bj k\bl}(p)\\
=&\vv<\nabla_{e_k}\nabla_{\ol{e_l}}e_i,\ol{e_j}>(p)-\vv<\nabla_{\ol{e_l}}\nabla_{e_k}e_i,\ol{e_j}>(p)-\vv<\nabla_{[e_k,\ol{e_l}]}e_i,\ol{e_j}>(p)\\
=&-\vv<\nabla_{\ol{e_l}}\nabla_{e_k}e_i,\ol{e_j}>(p)\\
=&-\ol{e_l}\vv<\nabla_{e_k}e_i,\ol{e_j}>(p)+\vv<\nabla_{e_k}e_i,\nabla_{\ol{e_l}}\ol{e_j}>(p)\\
=&-\ol{e_l}e_k(g_{i\bj})(p)+\ol{e_l}\vv<e_i,\nabla_{e_k}\ol{e_j}>(p)+\vv<\nabla_{e_k}e_i,\nabla_{\ol{e_l}}\ol{e_j}>(p)\\
=&-\ol{e_l}e_k(g_{i\bj})(p)+\vv<e_i,\nabla_{\ol{e_l}}\nabla_{e_k}\ol{e_j}>(p)+\vv<\nabla_{e_k}e_i,\nabla_{\ol{e_l}}\ol{e_j}>(p)\\
=&-\ol{e_l} e_k(g_{i\bj})+g^{\bmu\lam}e_k(g_{i\bmu})\ol{e_l}(g_{\lam\bj})
\end{split}
\end{equation}
where we have used Lemma \ref{lem-hol-con}, \eqref{eqn-lie-hol}, Lemma \ref{lem-christoffel} and Lemma \ref{lem-part-curv}.
\end{proof}

By Lemma \ref{lem-quasi-normal}, we have following direct corollary.
\begin{cor}\label{cor-curv-normal}
Let $(M,J,g)$ be an almost Hermitian manifold and $(e_1,e_2,\cdots,e_n)$ be a local quasi holomorphic normal frame at $p\in M$. Then
\begin{equation}
R_{i\bj k\bl}(p)=-\ol{e_l}{e_k}(g_{i\bj})(p).
\end{equation}
\end{cor}

\section{A generalization of Wu's result}
In this section, with the help the curvature formula derive in the last section, we obtain a generalization of Wu's result in \cite{Wu}.

\begin{thm}
Let $(M,J)$ be an almost complex manifold. Let $g,h$ be two almost Hermitian metrics
on $M$. Then
 $$R^{g+h}(X,\ol X,Y,\ol Y)\leq R^g(X,\ol X,Y,\ol Y)+R^h(X,\ol X,Y,\ol Y)$$
 for any two $(1,0)$-vectors $X$ and $Y$, where $R^{g+h},R^g$ and $R^h$ denote the curvature tensor of the metrics $g+h, g$ and $h$ respectively.
\end{thm}
\begin{proof}
Let $p\in M$, and $(e_1,e_2,\cdots,e_n)$ be a local quasi holomorphic normal frame at $p$ with respect to the almost Hermitian metric $g+h$.
Then, by Theorem \ref{thm-curv} and Corollary \ref{cor-curv-normal},
\begin{equation}
\begin{split}
&R^{g+h}_{i\bj k\bl}(p)\\
=&-\ol{e_l} e_k((g+h)_{i\bj})(p)\\
=&R^g_{i\bj k\bl}(p)+R^h_{i\bj k\bl}(p)-g^{\bmu\lam}e_k (g_{i\bmu})\ol{e_{l}} (g_{\lam\bj})(p)-h^{\bmu\lam}e_k (h_{i\bmu})\ol{e_{l}} (h_{\lam\bj})(p).
\end{split}
\end{equation}
Hence,
\begin{equation}
\begin{split}
&R^{g+h}(X,\ol X,Y,\ol Y)(p)\\
=&R^{g}(X,\ol X,Y,\ol Y)(p)+R^{h}(X,\ol X,Y,\ol Y)(p)\\
&-g^{\bmu\lam}X^i Y^k e_k(g_{i\bmu})\ol{X^jY^l}\ol{e_{l}} (g_{\lam\bj})(p)-h^{\bmu\lam}X^i Y^k e_k (h_{i\bmu})\ol{X^jY^l} \ol{e_{l}} (h_{\lam\bj})(p)\\
\leq&R^{g}(X,\ol X,Y,\ol Y)(p)+R^{h}(X,\ol X,Y,\ol Y)(p).\\
\end{split}
\end{equation}
\end{proof}
Similar as in \cite{Wu}, we have the following two direct corollaries.
\begin{cor}
Let $(M,J,g)$ be a compact almost Hermitian manifold with nonpositive (negative) holomorphic bisectional curvature. Then, there is an almost Hermitian metric on $M$ with nonpositive (negative) holomorphic bisectional curvature that is invariant under all the automorphisms of the almost complex structure $J$.
\end{cor}
\begin{defn}
Let $(M,J)$ be an almost complex manifold. Denote the collection of all almost Hermitian metrics with nonpositive holomorphic bisectional curvature as $\mathcal H(M)$.
\end{defn}
\begin{cor}
Let $(M,J)$ be an almost complex manifold with $\mathcal{H}(M)\neq\emptyset$.  Then $\mathcal H(M)$ is a convex cone.
\end{cor}
\section{A generalization of Zheng's result}
Let $\a$ be a smooth $(r,s)$ form on the almost complex manifold $(M,J)$. Recall that $\dbar \a=(d\a)^{(r,s+1)}$ and a $(r,0)$-form $\a$ is said to be a holomorphic $(r,0)$-form if $\dbar\a=0$.

\begin{lem}\label{lem-holo-lie}
A $(r,0)$-form $\a$ on an almost complex manifold is holomorphic if and only if $(L_{\ol X}\a)^{(r,0)}=0$ for any $(1,0)$ vector field $X$.
\end{lem}
\begin{proof}
By Cartan's formula,
\begin{equation}
\begin{split}
(L_{\ol X}\a)^{(r,0)}=(i_{\ol X}d\a+di_{\ol X}\a)^{(r,0)}=i_{\ol X}\dbar\a.
\end{split}
\end{equation}
The conclusion follows.
\end{proof}
\begin{defn}
A $(r,0)$-form $\a$ on an almost complex manifold is said to be pseudo holomorphic at $p\in M$ if $(L_{\ol X}\a)^{(r,0)}(p)=0$ for all $(1,0)$ vector field $X$.
\end{defn}
\begin{lem}\label{lem-crt-pseudo-holo-form}
Let $(M,J)$ be an almost complex manifold and $(e_1,e_2,\cdots,e_n)$ be a local pseudo holomorphic frame at $p\in M$. Let $(\omega^1,\omega^2,\cdots,\omega^n)$ be the dual frame of $(e_1,e_2,\cdots,e_n)$ and $\a$ be a $(r,0)$ form on $M$. Suppose that
\begin{equation}
\a=\sum_{i_1<i_2<\cdots<i_r}\a_{i_1i_2\cdots i_r}\omega^{i_1}\wedge\omega^{i_2}\wedge\cdots\wedge\omega^{i_r}.
\end{equation}
Then, $\alpha$ is pseudo holomorphic at $p$ if and only if $\bar v(\alpha_{i_1i_2\cdots i_r})(p)=0$ for all $v\in T^{1,0}_pM$ and any indices $i_1<i_2<\cdots<i_r$.
\end{lem}
\begin{proof}
Let $\ol{X}$ be a (1,0) vector field with $\ol{X}(p)=\bar v$. Then, for any indices $i_1<i_2<\cdots<i_r$,
\begin{equation}
\begin{split}
&(L_{\ol{X}})\a(e_{i_1},e_{i_2},\cdots,e_{i_r})(p)\\
=&\ol{X}(\a(e_{i_1},e_{i_2},\cdots,e_{i_r}))(p)-\a(L_{\ol{X}}e_{i_1}(p),e_{i_2}(p),\cdots,e_{i_r}(p))-\cdots\\
&-\a(e_{i_1}(p),e_{i_2}(p),\cdots,L_{\ol{X}}e_{i_r}(p))\\
=&\bar v(\a_{i_1i_2\cdots i_r})(p)-\a([\ol{X},e_{i_1}]^{(1,0)}(p),e_{i_2}(p),\cdots,e_{i_r}(p))-\cdots\\
&-\a(e_{i_1}(p),e_{i_2}(p),\cdots,[\ol{X},e_{i_r}]^{(1,0)}(p))\\
=&\bar v(\a_{i_1i_2\cdots i_r})(p).
\end{split}
\end{equation}
This completes the proof.
\end{proof}
\begin{thm}\label{thm-curv-2}
Let $(M,J,g)$ be an almost Hermitian manifold and $\a$ be a holomorphic $(1,0)$-form on $M$. Let $h$ be another metric on $M$ defined by
\begin{equation}
h(u,v)=g(u,v)+\a(u)\bar\a(v).
\end{equation}
for any two tangent vectors $u$ and $v$. Then $h$ is an almost Hermitian metric on $M$. Moreover,
\begin{equation}
R^h(X,\ol X,Y,\ol{Y})\leq R^g(X,\ol X,Y,\ol Y)
\end{equation}
for any $(1,0)$ vectors $X$ and $Y$.
\end{thm}
\begin{proof} It is easy to check that $h$ is an almost Hermitian metric by definition.

Fixed $p\in M$. Let $(e_1,e_2,\cdots,e_n)$ be a local quasi holomorphic normal frame at $p$ for the almost Hermitian metric $h$. Then, $h_{i\bj}=g_{i\bj}+\a_i\ol{\a_i}$. Since $\a$ is holomorphic on $M$, we have
\begin{equation}
\begin{split}
0=&\vv<L_{\ol{e_j}}\a,e_i>=\ol{e_j}\vv<\a,e_i>-\vv<\a,L_{\ol{e_j}}e_i>=\ol{e_j}(\a_i)-\vv<\a,\nabla_{\ol{e_j}}e_i>.\\
\end{split}
\end{equation}
Therefore,
\begin{equation}
\ol{e_j}e_k(\a_j)(p)=e_k\ol{e_j}(\a_i)(p)=e_k\vv<\a,\nabla_{\ol{e_j}}e_i>(p)=\vv<\a,\nabla_{e_k}\nabla_{\ol{e_j}}e_i(p)>=0\\
\end{equation}
by Lemma \ref{lem-part-curv} and $[\ol{e_j},e_k](p)=0$. Hence, by Theorem \ref{thm-curv} and Corollary \cite{cor-curv-normal},
\begin{equation}
\begin{split}
&R^h_{i\bj k\bl}(p)\\
=&-\ol{e_{l}}e_k(h_{i\bj})(p)\\
=&-\ol{e_{l}}e_k(g_{i\bj}+\alpha_i\ol{\a_j})(p)\\
=&-\ol{e_l}e_k(g_{i\bj})(p)-e_k(\alpha_i)\ol{e_l(\a_j)}(p)-\ol{e_l}{e_k}(\alpha_i)\ol{\a_j}(p)-\a_i\ol{e_l\ol{e_k}(\a_j)}(p)\\
=&R^g_{i\bj k\bl}(p)-g^{\bmu\lam}e_k(g_{i\bmu})\ol{e_{l}}(g_{\lam\bj})(p)-e_k(\alpha_i)\ol{e_l(\a_j)}(p).
\end{split}
\end{equation}
This completes the proof by processing the same as in the proof of Theorem \ref{thm-curv}.
\end{proof}
\begin{defn}
Let $(M,J)$ be a compact almost complex manifold. Denote the space of of holomorphic $(1,0)$-form on $M$ as $H^{1,0}(M)$.
\end{defn}
\begin{lem}\label{lem-double-holo-form}
Let $M^{2m}$ and $N^{2n}$ be two compact almost complex manifolds. Let $\phi_1,\phi_2,\cdots, \phi_r$ be a basis of $H^{1,0}(M)$ and
$\psi_1,\psi_2,\cdots,\psi_s$ be a basis of $H^{1,0}(N)$. Let $\rho$ be a holomorphic $(2,0)$-form on $M\times N$ and locally have the form of
\begin{equation}
\rho=\sum_{i=1}^m\sum_{j=1}^n\rho_{ij}\a^i\wedge\be^j
\end{equation}
 where $(\a^1,\a^2,\cdots,\a^m)$ and $(\be^1,\be^2,\cdots,\beta^n)$ are local (1,0)-frames on $M$ and $N$ respectively and $\rho_{ij}$'s are local smooth functions on $M\times N$. Then, there is a unique matrix $(a_{kl})$ of complex numbers with size $r\times s$ such that
 \begin{equation}
 \rho=\sum_{k=1}^r\sum_{l=1}^sa_{kl}\phi_k\wedge\psi_l.
 \end{equation}
\end{lem}
\begin{proof}
Fixed $y\in N$, let $(e_1,e_2,\cdots,e_n)$ be a local pseudo holomorphic frame at of $N$ at $y$ and $(\omega^1,\om^2,\cdots,\om^n)$ be its dual frame. Then,
\begin{equation}
\rho=\sum_{j=1}^n\theta_j(y)\wedge \omega^j(y)
\end{equation}
where $\theta_j(y)$'s are $(1,0)$-forms on $M=M\times\{y\}$. By Lemma \ref{lem-crt-pseudo-holo-form}, it is easy to check that $\theta_j(y)$ is a holomorphic $(1,0)$-form on $M$ for any $j$. Then
\begin{equation}
\theta_j(y)=\sum_{k=1}^rb_{kj}(y)\phi_k.
\end{equation}
Therefore
\begin{equation}
\rho=\sum_{k=1}^r\phi_k\wedge \sum_{j=1}^n b_{kj}(y)\omega^j(y).
\end{equation}
It easy to check that $\sum_{j=1}^n b_{kj}(y)\omega^j(y)$ is a holomorphic $(1,0)$ form on $N$. Hence
\begin{equation}
\sum_{j=1}^n b_{kj}(y)\omega^j(y)=\sum_{l=1}^sa_{kl}\psi_l
\end{equation}
where $a_{kl}$ are complex numbers. This completes the proof.
\end{proof}

With help of Lemma \ref{lem-double-holo-form} and the curvature formula in Theorem \ref{thm-curv}, the same argument as in \cite{Yu} give us the following
classification of almost Hermitian metrics on product of compact almost complex manifolds which generalizes the result of Zheng \cite{Zheng} and a previous result of the author \cite{Yu}.
\begin{thm}\label{thm-class-prod}
Let $M$ and $N$ be compact almost complex manifolds.
Let $\phi_1,\phi_2,\cdots, \phi_r$ be a basis of $H^{1,0}(M)$ and
$\psi_1,\psi_2,\cdots,\psi_s$ be a basis of $H^{1,0}(N)$. Then, for
any almost Hermitian metric $h$ on $M\times N$ with nonpositive holomorphic
bisectional curvature,
\begin{equation*}
\omega_h=\pi_1^{*}\omega_{h_1}+\pi_2^{*}\omega_{h_2}+\rho+\bar\rho
\end{equation*}
where $h_1$ and $h_2$ are almost Hermitian metrics on $M$ and  $N$ with
nonpositive holomorphic bisectional curvature respectively, $\pi_1$
and $\pi_2$ are natural projections from $M\times N$ to $M$ and from $M\times N$ to
$N$ respectively, and
\begin{equation*}
\rho=\sqrt{-1}\sum_{k=1}^{r}\sum_{l=1}^{s}a_{kl}\phi_k\wedge\ol{\psi_l}
\end{equation*}
with $a_{kl}$'s are complex numbers.
\end{thm}
Similarly as in \cite{Yu}, with the help of Theorem
\ref{thm-curv-2}, we can obtain the following corollary.
\begin{cor}
Let $M$ and $N$ be two almost Hermitian manifolds with $\mathcal H(M)\neq\emptyset$ and $\mathcal H(N)\neq\emptyset$. Then
$$\mbox{codim}_\R(\mathcal{H}(M)\times
\mathcal{H}(N),\mathcal{H}(M\times N))=2\dim H^{1,0}(M)\cdot
\dim H^{1,0}(N).$$
\end{cor}

\end{document}